\documentclass[a4paper]{article} 
\pdfoutput=1
\usepackage{multicol}
\usepackage{color}
\usepackage{amssymb,amsmath,verbatim,amsthm}
\usepackage[final]{graphicx}
\usepackage[active]{srcltx}
\usepackage{hyperref}

\newcommand{\R}{\mathbb{R}}
\newcommand{\Prob}{\mathbb{P}}
\renewcommand{\Pr}{\Prob}

\newcommand{\E}{\mathbb{E}}
\newcommand{\Ep}[1]{\E\left[ #1 \right]}

\newcommand{\Mc}{\mathcal{M}}

\newcommand{\eps}{\varepsilon}

\newcommand{\ps}{\pi^*}
\newcommand{\tm}{\tilde{m}}
\newcommand{\gp}{{\gamma_+}}
\newcommand{\gm}{{\gamma_-}}
\newcommand{\lmu}{\ell_{\mu}}
\newcommand{\rmu}{r_{\mu}}

\newtheorem{theorem}{Theorem}[section]
\newtheorem{prop}[theorem]{Proposition}
\theoremstyle{remark}
\newtheorem{remark}[theorem]{Remark}

\newcommand{\indic}[1]{\boldsymbol{1}_{\{\ensuremath{#1}\}}}

\newcommand{\as}{a.s.}

\newcommand{\eg}{e.g.}

\newcommand{\dx}{\mathrm{d}x}

\newcommand{\re}{\mathbb{R}} %
\newcommand{\td}{\mathrm{d}} %
\newcommand{\ub}{\ensuremath{\overline{b}}} %
\newcommand{\lb}{\ensuremath{\underline{b}}} %
\newcommand{\uB}{\ensuremath{\overline{B}}} %
\newcommand{\lB}{\ensuremath{\underline{B}}} %
\newcommand{\dbmps}{\indic{\overline{B}_\tau\geq \ub,\,
    \underline{B}_\tau> \lb}} %
\newcommand{\sS}{\overline{S}} %
\newcommand{\iS}{\underline{S}} %
\newcommand{\uG}{\overline{G}} %
\newcommand{\lG}{\underline{G}} %
\newcommand{\sM}{\overline{M}} %
\newcommand{\iM}{\underline{M}}%

\newcommand{\dbMps}{\indic{\overline{M}_\infty\geq \ub,\,
    \underline{M}_\infty> \lb}} %
\newcommand{\tntMevent}{\sM_{\infty} \ge \ub, \iM_{\infty} > \lb}
\newcommand{\tntM}{\indic{\tntMevent}}
\newcommand{\tntwTevent}{\overline{\omega}_{T} \ge \ub, \underline{\omega}_{T} > \lb}
\newcommand{\tntwT}{\indic{\tntwTevent}}
\newcommand{\rtntwTevent}{\overline{\omega}_T < \ub\textrm{ or
    }\underline{\omega}_T\leq \lb}
\newcommand{\rtntwT}{\indic{\rtntwTevent}}

\usepackage{asymptote}
\begin{asydef}
usepackage("amsmath");
texpreamble("\newcommand{\ub}{\ensuremath{\overline{b}}}\newcommand{\lb}{\ensuremath{\underline{b}}}");
\end{asydef}

\title{On joint distributions of the maximum, minimum and terminal value of a continuous uniformly integrable martingale
}

\author{Alexander~M.~G.~Cox\thanks{e-mail:
    \href{mailto:A.M.G.Cox@bath.ac.uk}{\nolinkurl{A.M.G.Cox@bath.ac.uk}};
    web: \url{www.maths.bath.ac.uk/\~mapamgc/}}\\
  Dept.\ of Mathematical Sciences\\
  University of Bath\\
  Bath BA2 7AY, UK \and Jan Ob\l \'oj\thanks{e-mail:
    \href{mailto:jan.obloj@maths.ox.ac.uk}{\nolinkurl{jan.obloj@maths.ox.ac.uk}}
    ; web: \url{www.maths.ox.ac.uk/\~obloj/}. 
    The research has received funding from the European Research Council under the European Union's Seventh Framework Programme (FP7/2007-2013) / ERC grant agreement no. 335421. The author is also grateful to the Oxford-Man Institute of Quantitative Finance and St John's College in Oxford for their support.}\\
  Mathematical Institute\\
  University of Oxford\\
  Oxford, OX2 6GG }

\date{\today}

\begin{document}
\maketitle
\begin{abstract}
  We study the joint laws of the maximum and minimum of a continuous, uniformly
  integrable martingale. In particular, we give explicit martingale inequalities
  which provide upper and lower bounds on the joint exit probabilities of a
  martingale, given its terminal law. Moreover, by constructing explicit and
  novel solutions to the Skorokhod embedding problem, we show that these bounds
  are tight. Together with previous results of Az\'ema \& Yor, Perkins, Jacka
  and Cox \& Ob\l\'oj, this allows us to completely characterise the upper and
  lower bounds on all possible exit/no-exit probabilities, subject to a given
  terminal law of the martingale. In addition, we determine some further
  properties of these bounds, considered as functions of the maximum and
  minimum.
\end{abstract}

\section{Introduction}

The study of the running maximum and minimum of a martingale has a prominent place in probability theory,  starting with Doob's maximal and $L^p$ inequalities. In seminal contributions, Blackwell and Dubins \cite{BD63}, Dubins and Gilat \cite{DubinsGilat:78} and Az\'ema and Yor \cite{AzemaYor:79,AzemaYor:79b} established that the distribution of the maximum $\sM_\infty:=\sup_{t\leq \infty} M_t$ of a uniformly integrable martingale $M$ is bounded from above, in stochastic order, by the so called Hardy-Littlewood transform of the distribution of $M_\infty$, and the bound is attained. This led to series of studies on the possible distributions of $(M_\infty,\sM_\infty)$ including Gilat and Meilijson \cite{GM88}, Kertz and R\"osler \cite{KR90, KertzRosler:92b, KertzRosler:93}, Rogers \cite{Rogers:93}, Vallois \cite{Vallois:93}, see also Carraro, El Karoui and Ob\l\'oj \cite{CarraroElKarouiObloj:09}. 

More recently, these problems have gained a new momentum from applications in the field of mathematical finance. The bounds on the distribution of the maximum, given the distribution of the terminal value, are interpreted as bounds on prices of barrier options given the prices of (vanilla) European options. Further, the bounds are often obtained by devising pathwise inequalities which then have the interpretation of (super) hedging strategies. This approach is referred to as robust pricing and hedging and goes back to Hobson \cite{Hobson:98b}, see also Ob\l\'oj \cite{Obloj:EQF} and Hobson \cite{Hobson:10} for survey papers. More recently, for example in Acciaio et.~al.~\cite{Acciaio:2013ab}, martingale inequalities have been used to study some classical probabilistic inequalities, and are of interest in their own right.

Here we propose to study the distribution of $(\sM_\infty,\iM_\infty)$, where $\iM_{\infty}:= \inf_{t \le \infty} M_t$ is the infimum of the process, given the distribution of $M_\infty$, for a uniformly integrable continuous martingale $M$. More precisely, we present sharp lower and upper bounds on all double exit/no-exit probabilities for $M$ in terms of the distribution of $M_\infty$, i.e.\ the probabilities that $\sM_\infty$ is \emph{greater/smaller} than $\ub$ \emph{and/or} that $\iM_\infty$ is \emph{greater/smaller} than $\lb$, for some barriers $\lb < \ub$. This amounts to considering eight different events. They of course come in pairs, \eg{} $\{\tntMevent\}$ is the complement of $\{\sM_\infty<\ub \textrm{ or }\iM_\infty\leq \lb\}$ and, by symmetry, it suffices to consider only one of $\{\tntMevent\}$ and $\{\sM_\infty<\ub, \iM_{\infty}\leq \lb\}$. It follows that to provide a complete description it suffices to consider the three events
\begin{equation}
\label{eq:eventsofinterest}
\{\sM_\infty\geq \ub, \iM_\infty\leq \lb\},\quad \{\sM_\infty< \ub, \iM_\infty>\lb\}\ \textrm{ and }\ \{\tntMevent\}.
\end{equation}
By continuity and time-change arguments, it follows that for a fixed distribution $\mu$ of $M_\infty$, our problem is equivalent to studying these events for $M_t=B_{t\wedge \tau}$ where $\tau$ varies among all stopping times such that $M$ is uniformly integrable and $M_\infty=B_\tau$ has distribution $\mu$, i.e.\ solutions to the Skorokhod embedding problem for $\mu$ in $B$, see Ob\l\'oj \cite{Obloj:04b}. Sharp bounds on the probability of the first event in \eqref{eq:eventsofinterest} follow from Perkins and tilted-Jacka solutions, see Section \ref{ap:mart} below. The case of the second event was treated in Cox and Ob\l\'oj \cite{Cox:2011aa} and is also recalled in Section \ref{ap:mart}.

Our contribution here is twofold. First, we derive lower and upper bounds on $\Pr(\tntMevent)$ in terms of the distribution of $M_\infty$ and give explicit constructions of martingales which attain the bounds. We do this by devising pathwise inequalities which give upper and lower bounds and then by constructing two new solutions to the Skorokhod embedding problem for which equalities are attained in our pathwise inequalities. Second, we study universal qualitative properties of the probabilities of the events in \eqref{eq:eventsofinterest} seen as surfaces in the parameters $\lb,\ub$. While the techniques used to derive the bounds on $\Pr(\tntMevent)$ are not new, the explicit constructions we need to use are novel, and our goal in the first part of the paper is to provide those bounds which are currently not known; in this sense, we complete previous work in the literature. The contribution in the second part of the paper is, to the best of our knowledge, the first attempt to address questions of this nature.

\subsection{Motivation}
We believe that there are two natural motivations for our results. First, we believe we solve an intrinsically interesting probabilistic question and second, our results correspond to robust pricing and hedging of certain double barrier options in finance. We elaborate now on both.

From the probabilistic point of view, we follow in the footsteps of seminal works mentioned above. The results therein were typically stated for a martingale and its maximum but naturally can be reformulated for a martingale and its minimum $\iM_\infty$. They grant us a full understanding of possible joint distributions of couples $(M_\infty,\sM_\infty)$ or $(M_\infty,\iM_\infty)$. In contrast, much less is known about the joint distribution of $(M_\infty,\sM_\infty,\iM_\infty)$ and it proves much harder to study (although promising recent progress has been made in this direction in a discrete time setting, when one considers the joint law of a random walk, its maximum, minimum and {\it signature} by \cite{Duembgen:2014aa}). Indeed, already in the case of Brownian motion $B$, while the distribution of $(B_t,\overline{B}_t)$ is readily accessible with a simply and explicit density, the distribution of the triplet $(B_t,\overline{B}_t,\underline{B}_t)$ is described through an infinite series. Likewise, $\Pr(\sM_\infty\geq \ub)$ is maximised among all martingales $M$ with a fixed distribution of $M_\infty$, by one extremal martingale simultaneously for all $\ub$. In contrast, as we will show here, maximising $\Pr(\tntMevent)$ will require martingales with qualitatively different behaviour for different values of $(\lb,\ub)$.

In terms of mathematical finance, the constructions presented here correspond to robust pricing (and hedging) of double touch/no-touch barrier options --- for a detailed discussion of applications we refer to our earlier papers \cite{Cox:2011ab, Cox:2011aa} where we studied the first two events in \eqref{eq:eventsofinterest}.  Such an option would pay out $1$ if and only if one barrier is attained \emph{and} a second given barrier is not attained, i.e.\ we consider the payoff of the form $\{\sS_T\geq \ub,\iS_T>\lb\}$, where $(S_t:t\leq T)$ is a uniformly integrable martingale representing the stock price process. The double touch/no-touch options are partially a theoretical construct --- (to the best of our knowledge) they are not commonly traded even in Foreign Exchange (FX) markets, where barriers options are most popular. However, they prove useful as they can be represented as a sum or difference of other barrier options. We can then interpret our results as super-/sub-hedges for sums and differences of barrier options. More precisely, we can write
\begin{eqnarray}
\indic{\sS_T\geq \ub,\iS_T>\lb} &=& \indic{\sS_T\geq \ub}-\indic{\sS_T\geq \ub,\iS_T\leq \lb} \label{eq:dtnt_decompose1}\\
  &=& 1-\Big(\indic{\iS_T\leq \lb}+ \indic{\sS_T <
    \ub,\iS_T > \lb}\Big)
  \label{eq:dtnt_decompose2}.
\end{eqnarray}
The first decomposition \eqref{eq:dtnt_decompose1} writes the payoff of a double touch/no-touch option as a difference of a one-touch option (with payoff $\indic{\sS_T\geq \ub}$) and a double touch option. The second decomposition \eqref{eq:dtnt_decompose2} writes the payoff of a double touch/no-touch option as one minus the portfolio of a one-touch option and a double no-touch (range) option with payoff $\indic{\sS_T< \ub,\iS_T> \lb}$. This is of particular interest as both one-touch and range options are liquidly traded in main currency pairs in FX markets. Effectively, using the no-arbitrage prices derived in Theorems \ref{thm:upper_price_mixed} and \ref{thm:lower_price_mixed} below, we obtain a way of checking for absence of arbitrage in the observed prices of European calls/puts, one-touch and range options. Furthermore, if one-touch options are liquidly traded, we can then exploit pathwise inequalities derived in this paper as super- or sub-hedging strategies for range options or double touch options. For certain barriers this will be sharper than the hedges derived in Cox and Ob\l\'oj \cite{Cox:2011ab, Cox:2011aa} which assumes only that vanilla options are liquid.

\subsection{Notation} \label{sec:notation}

Throughout the paper $M$ denotes a continuous uniformly integrable martingale and $B$ a standard real-valued Brownian motion. The running maximum and minimum of a Brownian motion $B$ or a martingale $M$ are denoted respectively $\uB_t=\sup_{u\leq t}B_u$ and $\lB_t=\inf_{u\leq t}B_u$, and similarly $\sM_t$ and $\iM_t$. The first hitting times of levels are denoted $H_x(B):=\inf\{t\geq 0: B_t=x\}$, $x\in \R$. Likewise we will consider $H_x(M)$ and $H_x(\omega)$, the first hitting times for a martingale $M$ and a continuous path $\omega$. Most of the time we simply write $H_x$ as it should be clear from the context which process/path we consider. We will use the hitting times primarily to express events involving the running maximum and minimum, e.g.\ note that $\dbmps=\indic{H_{\ub}\leq \tau < H_{\lb}}$ a.s.. We also introduce the following notation to indicate composition of stopping times: if $\tau_1, \tau_2$ are both stopping times, then the stopping time $(\tau_2 \circ \tau_1)(\omega) = \tau_1(\omega) + \tau_2(\theta_{\tau_1}(\omega))$, where $\theta_t(\omega)$ is the usual shift operator, $\theta_t:C(\R_+) \to C(\R_+)$ defined by $(\theta_t(\omega))_s = \omega_{t+s}$.

We use the notation $a \ll b$ to indicate that $a$ is \emph{much smaller} than $b$ -- this is only used to give intuition and is not rigorous. The minimum and maximum of two numbers are denoted $a\land b=\min\{a,b\}$ and $a\lor b=\max\{a,b\}$ respectively, and the positive part is denoted $a^+=a\lor 0$.

Finally, for a probability measure $\mu$ on $\re$ we let $-\infty\leq
\lmu<\rmu\leq \infty$ be the bounds of the support of $\mu$, i.e.\
$[\lmu,\rmu]$ is the smallest interval with
$\mu([\lmu,\rmu])=1$.

\section{Bounds for the probability of double exit/no-exit}
\label{sec:dtnt}

In this section we provide sharp bounds on the probability
\begin{equation*}
\Pr\left(\tntMevent\right)
\end{equation*}
where $\lb < 0 < \ub$, and $M=(M_t:t\leq \infty)$ is a continuous uniformly integrable martingale.  Our approach will involve two steps: first we provide pathwise inequalities which induce upper and lower bounds on the given event. Second, we show that these bounds are attained. More specifically, consider a continuous path $(\omega_t: 0\leq t\leq T)$, where $T\leq \infty$. We will introduce pathwise inequalities comparing $\tntwT$ to a sum of a ``static term," some function $f(\omega_T),$ and a ``dynamic term" of the generic form $\beta (\omega_T-b)\indic{H_b<T}$. Note that such a dynamic term is zero initially and, when $b$ is hit, it introduces a $\beta$-rotation of $f(\omega_T)$ around $b$. Note also that when evaluated on paths of a martingale, it will be a martingale. Consequently, we will construct random variables which dominate (or are dominated by) the random variable $\tntM$ and which can be decomposed into a martingale term and a function of the terminal value $M_\infty$. Bounds on the double exit/no-exit probability above will be obtained by taking expectations in these inequalities. We further claim that these bounds are tight. This is proven in the subsequent section, where we build extremal martingales by designing optimal solutions to the Skorokhod embedding problem for Brownian motion.

\subsection{Pathwise inequalities: upper bounds}
\label{sec:pathwise}

We need to consider three different inequalities. As we will see later, it is always optimal to use exactly one of them, and the choice depends on the distribution of $M_\infty$ and the values of $\ub,\lb$. We give the cases intuitive labels, their meaning will become clearer when we subsequently construct extremal martingales.  Throughout this and the next section we assume that $0<T\leq \infty$ is fixed and $(\omega_t: 0\leq t\leq T)$ is a given continuous function. The hitting times are relative to $\omega$.  To keep the notation simple we do not emphasise the dependence on $\omega$, e.g.\ $ H_{\lb}=H_{\lb}(\omega):= \inf\{t\leq T: \omega_t=\lb\}$, or $\uG^I(K)=\uG^I(K,(\omega_t:t\leq T))$.

\emph{$\uG^I$: upper bound for $\lb \ll 0<\ub$.}\\
The inequality is presented graphically in Figure~\ref{fig:dbmp_G1}. We can write it as:
\begin{eqnarray}
  \tntwT & \le & \frac{1}{(K-\lb)}\left((\omega_T - K)^+ - (\lb-\omega_T)^+ -
    (\omega_T - \lb) \indic{H_{\lb}<T}\right)+\indic{\omega_T>
    \lb}\nonumber \\
  & & {} =:  \uG^I(K), \label{eq:uG1def}
\end{eqnarray}
where we assume $K>\lb$. We include here the special case where $K=\infty$, which corresponds to the upper bound $\tntwT \le \indic{\omega_T \ge \lb}$.  Note that the coefficient $1/(K-\lb)$ is taken so that the right-hand side after rotation at time $H_{\lb}$ is zero above $K$.

\begin{figure}[th]
  \centering
  \begin{asy}[width=0.9\textwidth,height=0.4\textwidth,keepAspect=false]
    import graph;
    
    real xmin = -2;
    real xmax = 3;
    real eps = 0.25;

    pen p = black+1;
    pen q = black+0.75;

    draw((xmin,0)--(xmax+eps,0),q,Arrow);
    label("$B_{\tau}$",(xmax+eps,0),S);

    real lb = -1;
    real ub = 1;
    real K = 2;

    real slope = 1/(K-lb);
    
    pen p2 = deepblue+1+longdashdotted;

    draw((xmin,-slope*(lb-xmin))--(lb,0),p2,legend="$t \ge H_{\lb}$");
    draw((lb,1)--(K,1)--(xmax,1+(xmax-K)*slope),p2);

    pen p3 = deepgreen+1+dashed;

    draw((xmin,0)--(lb,0),p3,legend="$t < H_{\lb}$");
    draw((lb,1)--(K,0)--(xmax,0),p3);

    label("$\lb$",(lb,0),S);
    label("$\ub$",(ub,0),S);
    label("$K$",(K,0),S);

    real xlab = -1.8;
    real ylab = 1.8;
    real lablength = 0.8;
    real labskip = 0.3;
    real labextra = 0.2;

    frame fr = legend(invisible);

    fr = scale(0.75)*fr;

    add(fr,(xlab,ylab));
    
    roundbox(fr);

    // draw((xlab,ylab)--(xlab+lablength,ylab),p2);
    // label("$t \ge H_{\lb}$",(xlab+lablength+labextra,ylab),E,p);
    // draw((xlab,ylab+labskip)--(xlab+lablength,ylab+labskip),p3);
    // label("$t < H_{\lb}$",(xlab+lablength+labextra,ylab+labskip),E,p);
    
    clip(box((xmin-eps,-0.5-eps),(xmax+2*eps,2.25+eps)));

  \end{asy}
  \caption{\label{fig:dbmp_G1} $\uG^I(K)$ in \eqref{eq:uG1def} providing an upper bound for $\tntwT$}
\end{figure}

\emph{$\uG^{II}$: upper bound for $\lb< 0 < \ub$.}\\
This is a fairly simple case: if we hit neither $\lb$ nor $\ub$, the inequality is simply $0 \le \alpha_1(\omega_T - \lb)$ for some $\alpha_1 >0$, so that the value is $1$ if we strike $\ub$ initially, and $0$ if we strike $\lb$ initially. This strategy is illustrated in Figure~\ref{fig:dbmp_G2}. If the path hits either $\ub$ or $\lb$ we have a constant value of either $1$ or $0$ respectively:
\begin{eqnarray}
  \tntwT & \le & \alpha_1 \omega_T - \alpha_0
  -\alpha_1(\omega_T-\ub)\indic{H_{\ub}< H_{\lb} \wedge T} - \alpha_1 (\omega_T - \lb)
  \indic{H_{\lb}< H_{\ub} \wedge T} \nonumber \\
  & & {} =:  \uG^{II}. \label{eq:uG2def}
\end{eqnarray}
The constraints on $\alpha_0,\alpha_1$ correspond to the need for the
function to be zero if $\lb$ is struck first, and $1$ if $\ub$ is
struck first. We deduce that
\begin{equation}
  \label{eq:uG2def_par}
  \begin{split}
    \alpha_0 & =  \lb/(\ub-\lb) \\
    \alpha_1 & = 1/(\ub-\lb).
  \end{split}
\end{equation}

\begin{figure}[th]
  \centering
  \begin{asy}[width=0.9\textwidth,height=0.4\textwidth,keepAspect=false]
    import graph;
    
    real xmin = -2;
    real xmax = 3;
    real eps = 0.25;

    pen p = black+1;
    pen q = black+0.75;

    draw((xmin,0)--(xmax+eps,0),q,Arrow);
    label("$B_{\tau}$",(xmax+eps,0),S);

    real lb = -1;
    real ub = 1;
    // real K = 2;

    real slope = 1/(ub-lb);
    
    pen p2 = deepblue+1+longdashdotted;
    pen p3 = deepgreen+1+dashed;
    pen p4 = deepred + 1;

    draw((xmin,-slope*(lb-xmin))--(xmax,slope*(xmax-ub)+1),p2,legend="$t
    < H_{\lb}\wedge H_{\ub}$");
    // draw((lb,1)--(K,1)--(xmax,1+(xmax-K)*slope),p2);

    draw((xmin,0)--(xmax,0),p3,legend="$H_{\lb} \le H_{\ub} \wedge t$");
    draw((xmin,1)--(xmax,1),p4,legend="$H_{\ub} \le H_{\lb} \wedge t$");

    label("$\lb$",(lb,0),S);
    label("$\ub$",(ub,0),S);
    //label("$K$",(K,0),S);

    real xlab = -1.8;
    real ylab = 1.8;
    real lablength = 0.8;
    real labskip = 0.3;
    real labextra = 0.2;

    frame fr = legend(invisible);

    fr = scale(0.75)*fr;

    add(fr,(xlab,ylab));
    
    roundbox(fr);

    // draw((xlab,ylab)--(xlab+lablength,ylab),p2);
    // label("$t \ge H_{\lb}$",(xlab+lablength+labextra,ylab),E,p);
    // draw((xlab,ylab+labskip)--(xlab+lablength,ylab+labskip),p3);
    // label("$t < H_{\lb}$",(xlab+lablength+labextra,ylab+labskip),E,p);

    clip(box((xmin-eps,-0.5-eps),(xmax+2*eps,2.25+eps)));

  \end{asy}

    \caption{\label{fig:dbmp_G2} $\uG^{II}$ in \eqref{eq:uG2def} providing an upper bound for $\tntwT$}
\end{figure}

\emph{$\uG^{III}$: upper bound for $\lb< 0 \ll \ub$.}\\
The final inequality uses the fact that $\tntwT \le \indic{\overline \omega_T
  \ge \ub}$, and that the inequality for the latter also works for the former. We can then rewrite (2.2) from Brown,
Hobson and Rogers \cite{Brown:01b} as
\begin{equation}
  \label{eq:uG3def}
  \tntwT \leq \frac{(\omega_T-K)^+}{\ub-K} + \frac{\ub-\omega_T}{\ub-K}
  \indic{\overline \omega_T \geq \ub} =: \uG^{III}(K),
\end{equation}
where $K<\ub$.

\subsection{Pathwise inequalities:  lower bounds}
\label{sec:dtnt_sub}
Observe that we have $\tntwT = 1-\rtntwT$ a.s. It follows that a
pathwise upper bound for $\tntwT$ corresponds to a pathwise lower bound
of $\rtntwT$, and vice versa. We will use this below to rephrase some
of the lower bounds as upper bounds.

\emph{$\lG_I$: lower bound for $\lb< 0 \ll \ub$}.\\
We let $\lG_I$ to be the trivial inequality that the probability is
bounded below by zero:
$\lG_I\equiv 0$.

\emph{$\lG_{II}$: lower bound for $\lb<0<\ub$.}\\
We describe an upper bound for
$\rtntwT$ which, as argued above, is equivalent to a lower bound for $\tntwT$. 
The inequality depends on two parameters $K_1$ and $K_2$
where $K_1 \ge \ub > K_2 \ge \lb$. The construction starts with equality
on the region $[K_2,\ub)$ and inequality elsewhere. The first time the
path hits $\ub$, we rotate to get equality (with zero) on
$[K_1,\infty)$ and so that the value is exactly 1 at $\lb$. If the
path later hits $\lb$, we again rotate to gain equality (with 1) on
$(-\infty,K_2]$ and $[\ub,K_1]$.  We write it as an inequality
\begin{equation}\label{eq:dbmpH3}
  \begin{split}
    \rtntwT\leq \ &\alpha_2(K_2-\omega_T)^++(1-\alpha_4)\indic{\omega_T<
      \ub}-\alpha_2(\omega_T-\ub)^++\alpha_1(\omega_T-K_1)^++\alpha_4\\ &{}+
    \beta_1(\omega_T-\ub)\indic{H_{\ub}<H_{\lb}\land
       T}+\beta_2(\omega_T-\lb)\indic{H_{\ub}<H_{\lb}\leq T}\\
    &{}+\beta_3(\omega_T-\lb)\indic{H_{\lb}<H_{\ub}\land T}\\
    &{}=:1-\lG_{II}(K_1,K_2),
  \end{split}
\end{equation}
which we present graphically in Figure \ref{fig:dbmp_H3}. It follows that $\lG_{II}(K_1,K_2)$ is a lower bound for $\tntwT$. We deduce immediately from the rotation conditions that $\beta_1=\alpha_2-\alpha_1$, $\beta_2=\alpha_1$ and $\beta_3=\alpha_2$. We have to satisfy two more constraints, namely that after hitting $\ub$ and rotating the function is zero on $[K_1,\infty)$ and one at $\lb$. Working out the values we have
\begin{equation}
  \label{eq:valuesdbmpH3}
  \left\{ \begin{array}{l}
      \alpha_1=\frac{1}{K_1-\lb}\\
      \alpha_2=\frac{\ub-\lb}{(K_1-\lb)(\ub-K_2)}\\
      \alpha_4=\frac{K_1-\ub}{K_1-\lb}
    \end{array} \right. 
  \quad \left\{ \begin{array}{l}
      \beta_1=\alpha_2-\alpha_1\\
      \beta_2=\alpha_1\\
      \beta_3=\alpha_2
    \end{array} \right. . 
\end{equation}
Observe that $\alpha_4\in (0,1]$ and $0<\alpha_1\le\alpha_2$.  We note
that if we hit $\lb$ before $\ub$ we have a strict inequality in \eqref{eq:dbmpH3}. Also, in the case where $K_2 = \lb$ a number of the terms simplify: in particular, the construction initially gives $\lG_{III} = 1$ for $\omega_T \in
[\lb,\ub)$ for $T< H_{\ub}$. More generally, we can also have $K_1 =
\ub$ (with or without also $K_2 = \lb$) and all the claims remain
true.

\begin{figure}[th]
  \centering
  \begin{asy}[width=0.9\textwidth,height=0.4\textwidth,keepAspect=false]
    import graph;
    
    real xmin = -2;
    real xmax = 3;
    real eps = 0.25;

    pen p = black+1;
    pen q = black+0.75;

    draw((xmin,0)--(xmax+eps,0),q,Arrow);
    label("$B_{\tau}$",(xmax+eps,0),S);

    real lb = -1.5;
    real ub = 0.75;
    real K1 = 1.75;
    real K2 = -0.25;

    real alpha1 = 1/(K1-lb);
    real alpha2 = (ub-lb)/((ub-K2)*(K1-lb));
    real alpha4 = (K1-ub)/(K1-lb);

    pen p2 = deepblue+1+longdashdotted;
    pen p3 = deepgreen+1+dashed;
    pen p4 = deepred + 1;

    draw((xmin,1+alpha2*(K2-xmin))--(K2,1)--(ub,1),p4,legend= "$t < H_{\ub} \wedge H_{\lb}$");
    draw((ub,alpha4)--(K1,alpha4-alpha2*(K1-ub))--(xmax,alpha4-alpha2*(xmax-ub)+alpha1*(xmax-K1)),p4);

    draw((xmin,alpha1*(K1-xmin))--(K2,alpha1*(K1-K2))--(ub,1),p2,legend="$H_{\ub}\le
    t <  H_{\lb}$");
    draw((ub,alpha4)--(K1,0)--(xmax,0),p2);

    draw((xmin,1)--(K2,1)--(ub,1-alpha4),p3,legend="$H_{\ub}< H_{\lb} \le  t$");
    draw((ub,1)--(K1,1)--(xmax,1+alpha1*(xmax-K1)),p3);

    label("$\lb$",(lb,0),S);
    label("$\ub$",(ub,0),S);
    label("$K_2$",(K2,0),S);
    label("$K_1$",(K1,0),S);

    real xlab = 1.25;
    real ylab = 2.15;
    real lablength = 0.8;
    real labskip = 0.3;
    real labextra = 0.2;

    frame fr = legend(invisible);

    fr = scale(0.75)*fr;

    add(fr,(xlab,ylab));
    
    roundbox(fr);

    // draw((xlab,ylab)--(xlab+lablength,ylab),p2);
    // label("$t \ge H_{\lb}$",(xlab+lablength+labextra,ylab),E,p);
    // draw((xlab,ylab+labskip)--(xlab+lablength,ylab+labskip),p3);
    // label("$t < H_{\lb}$",(xlab+lablength+labextra,ylab+labskip),E,p);

    clip(box((xmin-eps,-0.5-eps),(xmax+2*eps,2.25+eps)));
  \end{asy}
  \caption{\label{fig:dbmp_H3} $(1-\lG_{II}(K_1,K_2))$ in \eqref{eq:dbmpH3}--\eqref{eq:valuesdbmpH3} providing an upper bound for $\rtntwT=1-\tntwT$. The case where we hit $\lb$ before $\ub$ is not shown.}
\end{figure}

\emph{$\lG_{III}$: lower bound for $\lb \ll 0<\ub$}.\\
As previously, we describe an upper bound for $\rtntwT$. 
The inequality is represented in Figure \ref{fig:dbmp_H2} and depends on two values $K_1$ and $K_2$ such that $\lb < K_2 < K_1 < \ub$. The inequality starts with equality (equal to 1) between $K_1$ and $\ub$, and if we hit $\ub$ initially, we rotate to get equality (to 0) between $K_2$ and $K_1$.  If we hit $\lb$ after this, we rotate again to ensure the function is equal to 1 below $K_2$. If we initially hit $\lb$ rather than $\ub$, we rotate to get a function that is generally strictly greater than one. We write it as
\begin{equation}\label{eq:dbmpH2}
  \begin{split}
    \rtntwT\leq \ &
    \alpha_2(K_2-\omega_T)^++\alpha_1(K_1-\omega_T)^++\indic{\omega_T <
      \ub}-\alpha_1(\omega_T-\ub)^+\\ &{} +
    \beta_1(\omega_T-\ub)\indic{H_{\ub}<H_{\lb}\land
      T}+\beta_2(\omega_T-\lb)\indic{H_{\ub}<H_{\lb}\leq T}\\ &{}
    +\beta_3(\omega_T-\lb)\indic{H_{\lb}<H_{\ub}\land T}\\ &{}
    =:1-\lG_{III}(K_1,K_2),
  \end{split}
\end{equation}
and it follows that $\lG_{III}(K_1,K_2)$ is a lower bound for $\tntwT$. We deduce immediately from the rotation conditions that $\beta_1=\alpha_1$, $\beta_2=\alpha_2$ and $\beta_3=\alpha_1+\alpha_2$. We have to satisfy two more constraints, namely that after hitting $\ub$ and rotating, the function is zero on $(K_2,K_1)$ and one in $\lb$. Working out the values we have
\begin{equation}
  \label{eq:valuesdbmpH2}
  \left\{ \begin{array}{l}
      \alpha_1=\frac{1}{\ub-K_1}\\
      \alpha_2=\frac{1}{K_2 - \lb}
    \end{array} \right. 
  \quad \left\{ \begin{array}{l}
      \beta_1=\alpha_1\\ \beta_2=\alpha_2\\ \beta_3=\alpha_1+\alpha_2
    \end{array} \right. . 
\end{equation}
As in the
previous case, we have a strict inequality in \eqref{eq:dbmpH2} if the path hits
$\lb$ before $\ub$.

\begin{figure}[th]
  \centering
  \begin{asy}[width=0.9\textwidth,height=0.4\textwidth,keepAspect=false]
    import graph;
    
    real xmin = -2;
    real xmax = 3;
    real eps = 0.25;

    pen p = black+1;
    pen q = black+0.75;

    draw((xmin,0)--(xmax+eps,0),q,Arrow);
    label("$B_{\tau}$",(xmax+eps,0),S);

    real lb = -1.5;
    real ub = 2.25;
    real K1 = 1;
    real K2 = 0.25;

    real slope = -1/(K2-lb);
    real slope2 = 1/(ub-K1);

    pen p2 = deepblue+1+longdashdotted;
    pen p3 = deepgreen+1+dashed;
    pen p4 = deepred + 1;

    draw((ub,1)--(K1,1)--(K2,1+slope2*(K1-K2))--(xmin,1+slope2*(K1-xmin)+slope*(xmin-K2)),p4,legend=
    "$t < H_{\ub} \wedge H_{\lb}$");
    draw((ub,0)--(xmax,-slope2*(xmax-ub)),p4);

    draw((xmin,-slope*(K2-xmin))--(K2,0)--(K1,0)--(ub,1),p2,legend="$H_{\ub}<
    t <  H_{\lb}$");
    draw((ub,0)--(xmax,0),p2);
    // draw((lb,1)--(K,1)--(xmax,1+(xmax-K)*slope),p2);

    draw((xmin,1)--(K2,1)--(K1,1-slope*(K1-K2))--(ub,1-slope*(ub-K2)+slope2*(ub-K1)),p3,legend="$H_{\ub}
    < H_{\lb} \le  t$");
    draw((ub,1-slope*(ub-K2))--(xmax,1-slope*(xmax-K2)),p3);

    label("$\lb$",(lb,0),S);
    label("$\ub$",(ub,0),S);
    label("$K_2$",(K2,0),S);
    label("$K_1$",(K1,0),S);

    real xlab = -1.9;
    real ylab = 2.15;
    real lablength = 0.8;
    real labskip = 0.3;
    real labextra = 0.2;

    frame fr = legend(invisible);

    fr = scale(0.75)*fr;

    add(fr,(xlab,ylab));
    
    roundbox(fr);

    // draw((xlab,ylab)--(xlab+lablength,ylab),p2);
    // label("$t \ge H_{\lb}$",(xlab+lablength+labextra,ylab),E,p);
    // draw((xlab,ylab+labskip)--(xlab+lablength,ylab+labskip),p3);
    // label("$t < H_{\lb}$",(xlab+lablength+labextra,ylab+labskip),E,p);

    clip(box((xmin-eps,-0.5-eps),(xmax+2*eps,2.25+eps)));
  \end{asy}

  \caption{\label{fig:dbmp_H2} $(1-\lG_{III}(K_1,K_2))$ in \eqref{eq:dbmpH2}--\eqref{eq:valuesdbmpH2} providing an upper bound for $\rtntwT=1-\tntwT$. The case when we hit $\lb$  before $\ub$ is not shown.}
\end{figure}

\subsection{Probabilistic bounds}
\label{sec:prob_bounds}

We now consider the pathwise inequalities above evaluated on a path of a continuous uniformly integrable martingale $M=(M_t:0\leq t\leq \infty)$. This gives a.s.\ bounds on $\dbMps$.
By taking expectations we obtain bounds on the double exit/no-exit probabilities in terms of the distribution of $M_\infty$. Indeed, observe that each of the bounds we get
can be decomposed into two terms. The first of these depends on $M_\infty$ alone, for example, in \eqref{eq:dbmpH3}, the sum of the four quantities preceded by an $\alpha$. The second corresponds to a martingale and disappears when taking expectations, e.g.\ considering again \eqref{eq:dbmpH3}, the three terms which are preceded by a $\beta$ sum to give a term with expected value zero.

\begin{prop}
  \label{prop:prob_upperbound}
  Suppose $M=(M_{t}: 0\leq t\leq \infty)$ is a continuous uniformly integrable martingale. Then
  \begin{equation}\label{eq:generalbound_dbmp}
    \Pr\left( \tntMevent \right) \leq \inf\left\{\Ep{ \uG^{I}(K)},\Ep{\uG^{II}},\Ep{\uG^{III}(K')}\right\},
  \end{equation}
  where the infimum is taken over $0<K'<\ub<K$ and where
  $\uG^{I},\uG^{II},\uG^{III}$ are given by
  \eqref{eq:uG1def},\eqref{eq:uG2def}--\eqref{eq:uG2def_par}, and
  \eqref{eq:uG3def} respectively, evaluated on paths of $M$.
\end{prop}

Our goal is to show that the above bound is optimal.  A key aspect of the above result is that the right hand-side of \eqref{eq:generalbound_dbmp} depends only on the distribution of $M_\infty$ and not on the law of the martingale $M$. We let $\mu$ be a probability measure on $\re$ with finite first moment. It is clear that we may then assume (subject to a suitable shift of the martingale) that the measure $\mu$ is centred. We also exclude the trivial case where $\mu = \delta_0$ from our arguments, so necessarily $\mu((-\infty,0))$ and $\mu((0,\infty))$ are both strictly positive.  We write $M\in \Mc_\mu$ to denote that $M$ is a continuous uniformly integrable martingale with $M_\infty\sim \mu$.

In the arguments below, we will commonly want to discuss the measure
$\mu$ restricted to some interval. Moreover, in the case where there
is an atom of $\mu$ at a point $y$, it may become necessary to split
the atom into more than one part. It will be convenient therefore to
split the measure $\mu$ according to its quantiles. We therefore
introduce the notation $F(x) = \mu((-\infty,x])$ for the usual
distribution function of the measure $\mu$, and write $F^{-1}(q) =
\inf \{x \in \R: F(x) \ge q\}\vee \lmu$. Then for $p,q \in [0,1]$ with
$p \le q$ we define the sub-probability measures
\begin{equation}
  \label{eq:2}
  \mu_p^q((-\infty,x]) = (F(x) \wedge q - p) \vee 0 =: F_p^q(x).
\end{equation}
In addition, we will write $\mu^q = \mu^q_0$ and $\mu_p =
\mu_p^1$. Observe that $\mu_p^q(\R) = q-p$.

The {\it barycentre} of $\mu$ associates to a non-empty Borel set
$\Gamma\subset \re$ the mean of $\mu$ over $\Gamma$ via
\begin{equation}\label{eq:barycentre}
  \mu_B(\Gamma)=\frac{\int_\Gamma u\, \mu(\td u)}{\int_\Gamma \mu(\td u)}. 
\end{equation}
An obvious extension is to consider the barycentre of the measure
$\mu$ when restricted to $\mu_p^q$, which we denote by $m_p^q$, so
\begin{equation}
  \label{eq:3}
  m_p^q = \begin{cases} (q-p)^{-1}\int x \, \mu_p^q(\dx) & \text{ if
    } q>p\\
    F^{-1}(q) & \text{ otherwise}
  \end{cases}.
\end{equation}

Now fix $\lb, \ub \in \R$ with $\lb < 0 < \ub$. Of importance in our
constructions will be the following notions. Given $p$ with $p\le
F(\lb-)$, we want to find the probability $q$ such that $m_p^q =
\lb$. Specifically, define a function $\rho_-: [0,F(\lb-)] \to
[F(\lb),1]$ by
\begin{equation}
  \label{eq:4}
  \rho_-(p) = \inf \{ q \ge F(\lb) : m_p^q \ge \lb\}.
\end{equation}
Similarly, we can define $\rho_+: [F(\ub),1] \to
[0,F(\ub-)]$ by
\begin{equation}
  \label{eq:5}
  \rho_+(q) = \sup \{ p \le F(\ub-) : m_p^q \le \ub\}.
\end{equation}
It is straightforward to see that $\rho_-(p)$ and $\rho_+(q)$ are both continuous, strictly decreasing functions, and are well defined since $\lb < 0 = \int x \, \mu(\dx) < \ub$, so that the infimum in \eqref{eq:4} and the supremum in \eqref{eq:5} are both over non-empty sets. Further, note that we get:
\begin{equation}
  \label{eq:6}
  m_p^{\rho_-(p)} = \lb, m_{\rho_+(q)}^q = \ub
\end{equation}
for all $p \le F(\lb-)$ and all $q \ge F(\ub)$. Observe that the barycentre has two nice properties: first, if we rescale the measure $\mu$ by a constant, then the barycentre is unchanged. Second, if we wish to show that a measure $\mu$ has barycentre $b$, it is sufficient to show that
\begin{equation*}
  \int (x-b) \, \mu(dx) = 0,
\end{equation*}
independent of whether $\mu$ is a probability measure. In the case where $\mu$ is a probability measure $\mu_{B}(\R)$ is just the mean of the measure. Finally, we introduce the additional useful notation
\begin{equation*}
  \tm_{p}^q = (q-p) m_p^q.
\end{equation*}

Since the functions $\rho_+$ and $\rho_-$ are both continuous and strictly decreasing, their inverses are also continuous and strictly decreasing where defined --- for example, $\rho_+^{-1}$ maps $[\rho_+(1),F(\ub-)] \to [F(\ub),1]$.

A critical role in the construction of embeddings will be played by the following definition. Set
\begin{equation}
  \label{eq:1}
  \ps = \inf \left\{ p \in [\rho_+(1)\vee F(\lb),F(\ub-)] :
  \rho_+^{-1}(p)-p \le \frac{-\lb}{\ub-\lb}\right\} \wedge F(\ub-),
\end{equation}
where we use the standard convention that the infimum of an empty set is $\infty$. Since $\rho_+^{-1}(F(\ub-))= F(\ub)$, $\rho_+^{-1}(p)$ is continuous and $\lb < 0$, it follows that $\ps\in [\rho_+(1)\vee F(\lb),F(\ub-)]$.  Then we have the following theorem.
\begin{theorem}\label{thm:upper_price_mixed}(Upper bound)
The bound in \eqref{eq:generalbound_dbmp} is sharp. More precisely, let $\mu$ be a given centred probability measure on $\re$. Then exactly one of the following is true
  \begin{enumerate}
  \item[\fbox{I}] `$\lb \ll 0<\ub$': we have $\ps =F(\lb)$ and $\rho_+^{-1}(\ps) -\ps <
    -\lb(\ub-\lb)^{-1}$.\\
    Then there is a martingale $M\in \Mc_\mu$ such that
    \begin{equation*}
      \Pr(\tntMevent) =\Ep{ \uG^I(z^*)},
    \end{equation*}
    where $\uG^I$ is given by \eqref{eq:uG1def} evaluated on paths of $M$, and $z^* =
    F^{-1}(\xi)$ where $\xi$ solves
    \begin{equation} \label{eqn:xidefn} 
      \int (x-\lb) \, \mu_{F(\lb)}^\xi = -\lb.
    \end{equation}
    
  \item[\fbox{II}] `$\lb<0<\ub$':  we have $\rho_+^{-1}(\ps) -\ps \ge -\lb(\ub-\lb)^{-1}$.\\ 
    Then there is a martingale $M\in \Mc_\mu$ such that
        \begin{equation*}
      \Pr(\tntMevent) =\Ep{ \uG^{II}} = -\lb(\ub-\lb)^{-1},
    \end{equation*}
    where $\uG^{II}$ is given by
    \eqref{eq:uG2def}--\eqref{eq:uG2def_par} evaluated on paths of $M$.

  \item[\fbox{III}] `$\lb< 0 \ll \ub$': we have $\ps = \rho_+(1)$ and $1-\ps < -\lb(\ub-\lb)^{-1}$.  \\ 
    Then there is a martingale $M\in \Mc_\mu$ such that
    \begin{equation*}
      \Pr(\tntMevent) = \Ep{\uG^{III}(F^{-1}(\ps))},
    \end{equation*}
    where $\uG^{III}$ is given by \eqref{eq:uG3def} evaluated on paths of $M$.
  \end{enumerate}
\end{theorem}

In a similar manner to Proposition \ref{prop:prob_upperbound}, the
pathwise inequalities described in Section~\ref{sec:dtnt_sub} instantly imply
a lower bound on the double exit/no-exit probabilities:
\begin{prop}
  \label{prop:dtnt_lowerbound}
    Suppose $M=(M_{t}: 0\leq t\leq \infty)$ is a continuous uniformly integrable martingale. Then
  \begin{equation}\label{eq:generalbound_dbmp2}
    \Pr\left( \tntMevent\right) \geq \sup\left\{0,\Ep{\lG_{II}(K_1',K_2)},\Ep{ \lG_{III}(K_1,K_2)}\right\},
  \end{equation}
  where the supremum is taken over $\lb<K_2<K_1<\ub<K_1'$ and where
  $\lG_{II},\lG_{III}$ are given by
  \eqref{eq:dbmpH3}, \eqref{eq:valuesdbmpH3} and
  \eqref{eq:dbmpH2}, \eqref{eq:valuesdbmpH2} respectively, evaluated on paths of $M$.
\end{prop}

We proceed to show that this lower bound is optimal. Write
\begin{equation}
  \label{eq:16}
  \gamma = 1-F(\ub-) + F(\lb),
\end{equation}
and consider the condition
\begin{equation}
  \label{eq:7}
  \tm_{F(\lb)}^{F(\ub-)} + \gamma \lb \ge 0.
\end{equation}
If this holds, then we can find $\lambda \in (F(\lb),F(\ub-)]$ such
that
\begin{equation}
  \label{eq:9}
  \tm_{F(\lb)}^{\lambda} + (1-\lambda + F(\lb)) \lb = 0
\end{equation}
since the left-hand side is increasing in $\lambda$ and runs between
$\lb$ and a term which is positive by \eqref{eq:7}. If \eqref{eq:7}
fails, we can imagine moving mass from an atom at $\lb$, to the right,
in the process moving the average of the mass upwards. In this case,
consider the condition
\begin{equation}
  \label{eq:8}
    \tm_{F(\lb)}^{F(\ub-)} + \gamma \ub \le 0.
\end{equation}
If \eqref{eq:7} fails, and \eqref{eq:8} holds, then we set $\xi =
F(\lb)$ and we can find $\lambda \in (0,\gamma]$ such that
\begin{equation}
  \label{eq:10}
  \tm_{F(\lb)}^{F(\ub-)} + \lambda \lb + (\gamma-\lambda) \ub = 0.
\end{equation}
Given such a $\lambda$, we will show that there exists $\ps \in
[F(\ub-),1)$ such that
\begin{equation}
  \label{eq:11}
  \tm^{\xi} + \tm_{F(\ub-)}^{\ps} = \lb ( \xi + \ps - F(\ub-)).
\end{equation}

If \eqref{eq:8} also fails, and 
\begin{equation}
  \label{eq:12}
  \text{either } \rho_-(0) \ge F(\ub-) \text{ or }
    \rho_-(0) < F(\ub-) \text{ and }
  \tm_{\rho_-(0)}^{F(\ub-)} + \ub ( 1-F(\ub-) + \rho_-(0)) >0
\end{equation}
then there exists $\xi \in (F(\lb),\rho_-(0)\wedge F(\ub-))$ such that
\begin{equation}
  \label{eq:13}
  \tm_{\xi}^{F(\ub-)} + \ub (1-F(\ub-) + \xi) = 0.
\end{equation}
Then we define $\ps$ as the solution to \eqref{eq:11} again. 

Finally, if \eqref{eq:7}, \eqref{eq:8} and \eqref{eq:12} all fail,
then there exists $\ps \in [\rho_-(0),F(\ub-))$ such that
\begin{equation}
  \label{eq:14}
  \tm_{\ps}^{F(\ub-)} + \ub(1-F(\ub-) + \ps) =0.
\end{equation}

\begin{theorem}\label{thm:lower_price_mixed}(Lower bound)
The bound in \eqref{eq:generalbound_dbmp2} is sharp. More precisely, let $\mu$ be a given centred probability measure on $\re$. Then exactly one of the following is true:
  \begin{enumerate}
  \item[\fbox{I}] `$\lb< 0 \ll \ub$': condition \eqref{eq:7} holds.\\
    Then there is a martingale $M\in \Mc_\mu$
    such that $\Pr(\tntMevent) = 0 = \Ep{\lG_{I}}$.

  \item[\fbox{II}] `$\lb<0<\ub$': condition \eqref{eq:7} fails, and either \eqref{eq:8} holds or
    \eqref{eq:8} fails and \eqref{eq:12} holds.\\
    Then there is a martingale $M\in \Mc_\mu$ such that
      \begin{equation}
        \label{eq:15}
        \Pr(\tntMevent) = \Ep{\lG_{II}(\ps,\xi)},
      \end{equation}
      where $\lG_{II}$ is given via \eqref{eq:dbmpH3} and
      \eqref{eq:valuesdbmpH3}, evaluated on paths of $M$, and $\ps$ solves \eqref{eq:11}.

    \item[\fbox{III}] `$\lb \ll 0<\ub$': conditions \eqref{eq:7}, \eqref{eq:8} and \eqref{eq:12} fail. \\
      Then there is a martingale $M\in \Mc_\mu$ such that
      \begin{equation}
        \label{eq:17}
        \Pr(\tntMevent) = \Ep{\lG_{III}(\ps,\rho_-(0))}
      \end{equation}
      where $\lG_{II}$ is given via \eqref{eq:dbmpH2} and
      \eqref{eq:valuesdbmpH2}, evaluated on paths of $M$, and $\ps$ is given by \eqref{eq:14}.
    \end{enumerate}
  \end{theorem}

\begin{remark}
  Throughout the paper, we have assumed that $(M_t)_{t \ge 0}$ has
  continuous paths. This assumption can be relaxed. It
  is relatively simple to see that if we only assume that barriers
  $\lb,\ub$ are crossed in a continuous manner then all of our results
  remain true. If we only assume that $(M_t)$ has c\`adl\`ag paths
  then the situation is more complex. The optimal behaviour
  will essentially be as before, but we can use jumps to hide some of
  the occasions where a barrier is hit. More precisely, consider the
  continuous martingale $M$ given in Theorem~\ref{thm:upper_price_mixed} and, for $\eps > 0$, consider
  the time-change:
  \[
  \rho^\eps_t = \inf\{u\ge t : M_u \in [\lb+\eps,\infty)\}.
  \]
  Then $N_t = M_{\rho^\eps_t}$ is a UI martingale which excludes paths
  of $M_t$ where the minimum goes below $\lb+\eps$, but which later
  return above $\lb+\eps$. In general, any possible martingale $M_t$
  can be improved by performing such an operation, and so this
  suggests that an optimal discontinuous model can be chosen in such a
  manner that it is continuous on $[\lb+\eps,\infty)$ and only takes
  values on $(-\infty,\lb]$ if it is the final value of the
  martingale. This observation can be used as a starting point for an
  analysis similar to that given above to determine the optimal
  martingale models for a given measure. We do not pursue the details here.
\end{remark}

\section{Proofs that the bounds are sharp via new solutions to the Skorokhod embedding problem}
\label{sec:proofs}

In this section we prove Theorems~\ref{thm:upper_price_mixed} and \ref{thm:lower_price_mixed}. We do this by constructing new solutions to the Skorokhod embedding problem for a Brownian motion $B$.  Specifically, we will construct stopping times $\tau$ such that $B_{\tau} \sim \mu$, $(B_{t \wedge \tau}:t\geq 0)$ is UI and equalities are attained almost surely in the inequalities of Sections \ref{sec:pathwise}--\ref{sec:dtnt_sub}. It is then straightforward to see that martingales required in Theorems~\ref{thm:upper_price_mixed} and \ref{thm:lower_price_mixed} are given by $M_t:= B_{t \wedge \tau}$.

We will use below some well known facts about the existence of Skorokhod embeddings. Specifically, given a measure $\mu$ with mean $m$ and a Brownian motion $B$ with $B_0 = m$, then there exists a stopping time $\tau$ such that $B_{\tau} \sim \mu$ and $(B_{t \wedge \tau}: t \geq 0)$ is uniformly integrable. Moreover, it follows from uniform integrability that if the measure $\mu$ is supported on a bounded interval, then the process will stop before the first exit time of the interval.

\begin{proof}[Proof of Theorem \ref{thm:upper_price_mixed}]

  We take $B=(B_t:t\geq 0)$ a standard real-valued Brownian motion. All the hitting times $H_{\bullet}$ below are for $B$.  As described above, we will prove this result by constructing a stopping time $\tau$ such that $B_{\tau}$ has the distribution $\mu$, and such that the conjectured bounds hold for the corresponding continuous time martingale which is the stopped process.

  From the definition of $\ps$ in \eqref{eq:1} it is clear that at least one of the cases holds. Clearly \fbox{II} excludes the other two. To show that \fbox{I} and \fbox{III} are exclusive, as $\rho_+^{-1}(\rho_+(1))=1$, it suffices to argue that the following is impossible
  \begin{equation}\label{eqn:imposs_for_y}
    \ps = \rho_+(1) = F(\lb) > \ub(\ub-\lb)^{-1}.
  \end{equation}
  Assume \eqref{eqn:imposs_for_y} holds. From the last condition we get $\ub (1-\ps) < - \lb \ps$, and using the fact that $\ps = \rho_+(1)$, this can be expressed as $\int x \,\mu_{\ps}(\dx) + \lb \ps < 0$. However $\ps \ge F(\lb)$ implies that this is greater than or equal to $\int x \, \mu(\dx) = 0$ giving a contradiction. We conclude that the cases \fbox{I}, \fbox{II} and \fbox{III} are exclusive.

  We now show the existence of a suitable embedding. We consider
  initially the case \fbox{I}. We first note that the solution $\xi$
  of \eqref{eqn:xidefn} is in $(\rho_+^{-1}(\ps),1]$. Since
  \begin{eqnarray*}
    \int (x - \lb) \, \mu_{F(\lb)}^{\rho_+^{-1}(F(\lb))}(\dx) & = &
    \int (x - \ub) \mu_{F(\lb)}^{\rho_+^{-1}(F(\lb))}(\dx) +
    \int (\ub - \lb) \, \mu_{F(\lb)}^{\rho_+^{-1}(F(\lb))}(\dx) \\
    & = & (\ub - \lb)\left( \rho_+^{-1}(F(\lb))-F(\lb)\right) \\
    & < & - \lb,
  \end{eqnarray*}
  we conclude that $\xi > \rho_+^{-1}(\ps)$. To see that $\xi \le 1$,
  we note:
  \begin{equation*}
    \int (x - \lb) \, \mu_{F(\lb)} (\dx) \ge \int (x - \lb) \, \mu
    (\dx) = -\lb.
  \end{equation*}
  Since the expression $\int (x-\lb) \, \mu_{F(\lb)}^\xi(\dx)$ is strictly increasing and continuous in $\xi$, there is a unique $\xi$. For this value of $\xi$, we now define a measure $\nu$ by
  \begin{equation*}
    \nu = \left[ -\frac{\lb}{\ub-\lb} - (\xi - F(\lb))\right] \delta_{\lb} + \mu_{F(\lb)}^\xi. 
  \end{equation*}
  Observe that the atom at $\lb$ has mass greater than or equal to
  zero, and by construction, $\nu$ has total mass $-\lb(\ub-\lb)^{-1}$
  and barycentre $\ub$ since
  \begin{eqnarray*}
    \int (x-\ub) \, \nu(\dx) & = &  \int (x-\ub) \, \mu_{F(\lb)}^\xi(\dx)
    + \left[-\frac{\lb}{\ub-\lb} - (\xi - F(\lb))\right] (\lb-\ub)\\
    & = & (\lb-\ub)(\xi-F(\lb)) - \lb + \left[-\frac{\lb}{\ub-\lb} -
      (\xi - F(\lb))\right] (\lb-\ub)\\
    & = & 0.
  \end{eqnarray*}
    
  We now show that this means we can construct a suitable embedding. The idea will be initially to run until the first time we hit either of $\ub$ or $\lb$. The mass that hits $\ub$ first will then be used to embed $\nu$, and all the mass that hits $\lb$ (which will include the atomic term from $\nu$) can then be embedded in the remaining areas, $(0,\lb] \cup [F^{-1}(\xi),\infty)$. So suppose we are in case \fbox{I}, and let $\tau_1$ be first time we hit one of $\lb$ or $\ub$, so $\tau_1 = H_{\ub} \wedge H_{\lb}$. Then $\Pr(B_{\tau_1} = \ub) = -\lb(\ub-\lb)^{-1}$. Let $\tau_2$ be a UI embedding of the probability measure $-\frac{\ub-\lb}{\lb}\nu$ given $B_0 = \ub$ and let $\tau_3$ be a UI embedding of $\sigma$ given $B_0 = \lb$, where
  \begin{equation*}
    \sigma = \frac{\left(\mu^{F(\lb)}+\mu_{\xi}\right)}{F(\lb)+1-\xi}.
  \end{equation*}
  It can be verified that $\sigma$
  has 
  barycentre $\lb$ since
  \begin{equation*}
    \int (x-\lb) \left(\mu^{F(\lb)}+\mu_{\xi}\right)(\dx) = \int
    (x-\lb) \, \mu(\dx) - \int (x-\lb) \, \mu_{F(\lb)}^\xi(\dx) = 0.
  \end{equation*}
  Then (recalling the definition in Section~\ref{sec:notation}) we set
  \begin{eqnarray*}
    \tau & := & \tau_2 \circ \tau_1\indic{\tau_1 = H_{\ub}}
    \indic{\tau_2 \circ \tau_1 < H_{\lb}}\\
    && {} +  \tau_3 \circ \tau_1 \indic{\tau_1 = H_{\lb}}\\
    && {} +  \tau_3 \circ \tau_2 \circ \tau_1 \indic{\tau_1 = H_{\ub}}
    \indic{\tau_2 \circ \tau_1 = H_{\lb}}.
  \end{eqnarray*}
  We see that $\tau$ is a UI embedding of $\mu$, and moreover $\tau$
  is such that $ \dbmps= \uG^I(F^{-1}(\xi))$ a.s..

  Consider now case \fbox{II}. Suppose initially that in addition, $\rho_+^{-1}(\ps)-\ps = -\lb (\ub-\lb)^{-1}$. We define measures $\nu$ and $\sigma$ by:
  \begin{eqnarray*}
    \nu & = & \frac{1}{\rho_+^{-1}(\ps)-\ps} \mu_{\ps}^{\rho_+^{-1}(\ps)}\\
    \sigma & = & \frac{1}{1+\ps-\rho_+^{-1}(\ps)}
    \left(\mu^{\ps} + \mu_{\rho_+^{-1}(\ps)}\right).
  \end{eqnarray*}
  Then $\nu$ has barycentre $\ub$, while $\sigma$ has barycentre
  $\lb$. Let $\tau_1$ be as above, $\tau_2$ be a UI embedding of
  $\nu$ given $B_0 = \ub$ and $\tau_3$ be a UI embedding of $\sigma$
  given $B_0 = \lb$. Then the stopping time
  \begin{eqnarray*}
    \tau & := & \tau_2 \circ \tau_1 \indic{\tau_1 = H_{\ub}} \\
    && {} + \tau_3 \circ \tau_1 \indic{\tau_1 = H_{\lb}}
  \end{eqnarray*}
  is a UI embedding of $\mu$, and $B_{t\wedge \tau}$ satisfies
  $\dbmps=\uG^{II}$ \as{} where $\uG^{II}$ is the random variable
  defined in \eqref{eq:uG2def}, evaluated on paths of $B$.

  The case where $\rho_+^{-1}(\ps)-\ps > -\lb (\ub-\lb)^{-1}$ is almost identical --- observe that in this case, there must be an atom of $\mu$ at $\ub$ with $F(\ub)-F(\ub-) > -\lb (\ub-\lb)^{-1}$. However, the argument above works without alteration if we take:
  \begin{eqnarray*}
    \nu & = & \frac{1}{-\lb (\ub-\lb)^{-1}} \delta_{\ub}\\
    \sigma & = & \frac{1}{1+ \lb (\ub-\lb)^{-1}}(\mu-\nu).
  \end{eqnarray*}

  Finally we consider \fbox{III}. Then define measures $\nu$ and
  $\sigma$ by:
  \begin{eqnarray*}
    \nu & = & \frac{1}{1-\ps} \mu_{\ps}\\      
    \sigma & = & \frac{1}{\ps}\mu^{\ps}.
  \end{eqnarray*}
  So the barycentre of $\nu$ is $\ub$, and the
  barycentre of $\sigma$ is $m^{\ps}$. Define $\tau_1$ to be the first
  hitting time of $\{m^{\ps},\ub\}$, so $\tau_1 = H_{m^{\ps}}\wedge H_{\ub}$, then $\Pr(B_{\tau_1} = \ub) =
  \ps=- m^{\ps}(\ub - m^{\ps})^{-1}$. We may then proceed as
  above, so we define $\tau_2$ to be a UI embedding of $\nu$ given
  $B_0 = \ub$ and $\tau_3$ to be a UI embedding of $\sigma$ given $B_0
  = m^{\ps}$. Then the stopping time
  \begin{eqnarray*}
    \tau & := & \tau_2 \circ \tau_1 \indic{\tau_1 = H_{\ub}} \\
    && {} + \tau_3 \circ \tau_1 \indic{\tau_1 = H_{m^{\ps}}}
  \end{eqnarray*}
  is a UI embedding of $\mu$, and satisfies $\dbmps=\uG^{III}(F^{-1}(\ps))$
  \as{} where $\uG^{III}(\cdot)$ is the random variable defined in
  \eqref{eq:uG3def}, evaluated on paths of $B$.
\end{proof}

\begin{proof}[Proof of Theorem \ref{thm:lower_price_mixed}]
The setup, and general methodology, is analogous to the proof of Theorem \ref{thm:upper_price_mixed} above.

  It follows from their respective definitions that exactly one of
  \fbox{I}, \fbox{II} and \fbox{III} holds.

  Suppose \fbox{I} holds, so that \eqref{eq:7} is true. Then, by
  continuity, there exists $\lambda\in (F(\lb),F(\ub-)]$ such that
  \eqref{eq:9} holds (taking $\lambda = F(\lb)$ gives $\lb$ on the
  left hand side of \eqref{eq:9}). Let $\tau_1$ be a UI
  embedding of
  \begin{equation}
    \label{eq:18}
    \chi = \mu_{F(\lb)}^\lambda + (1-\lambda + F(\lb))\delta_{\lb}
  \end{equation}
  in the Brownian motion starting at 0, and observe that the measure
  \begin{equation*}
    \nu = \frac{\mu^{F(\lb)} + \mu_{\lambda}}{1-\lambda + F(\lb)}
  \end{equation*}
 has mean $\lb$, which follows since:
 \begin{align*}
   (1-\lambda + F(\lb)) \int x \nu(\dx) & = \tm^{F(\lb)} +
   \tm_{\lambda}\\
   & = -\tm_{F(\lb)}^{\lambda} = \lb(1-\lambda + F(\lb)).   
 \end{align*}
 Let $\tau_2$ be a UI embedding of $\nu$ in a Brownian motion starting
 from $B_0=\lb$. Finally define
  $$\tau:=\tau_1\indic{B_{\tau_1}\neq
    \lb}+\tau_2\circ\tau_1\indic{B_{\tau_1}=\lb},$$ which is a UI
  embedding of $\mu$ in the Brownian motion $B$. Note that
  $\overline{B}_\tau\geq \ub$ only if $\underline{B}_\tau\leq \lb$. It
  follows that $\dbmps=0=\lG_{I}$ a.s. 

  Suppose now that \fbox{II} holds. We consider separately the case
  where \eqref{eq:7} fails and \eqref{eq:8} holds, and the case where
  both \eqref{eq:7} and \eqref{eq:8} fail, but \eqref{eq:12}
  holds. First suppose \eqref{eq:8} holds. Then
  \begin{equation*}
    \lambda  \mapsto \tm_{F(\lb)}^{F(\ub-)} + \lambda \lb + (\gamma -
    \lambda) \ub
  \end{equation*}
  is continuous, and strictly negative for $\lambda = 0$ and positive for
  $\lambda = \gamma$. Hence there exists $\lambda \in (0,\gamma]$ such
  that \eqref{eq:10} holds. Fix $\xi = F(\lb)$ and consider
  \begin{equation*}
    [F(\ub-),1) \ni \ps \mapsto \tm^{\xi} + \tm_{F(\ub-)}^{\ps} - \lb ( \xi + \ps - F(\ub-)).
  \end{equation*}
  In the limit as $\ps \to 1$, the expression simplifies to
  $-\tm^{F(\ub-)}_{F(\lb)}-\gamma\lb$ which is strictly positive since
  \eqref{eq:7} is assumed to fail, while if $\ps = F(\ub-)$ the
  expression simplifies to $\tm^{F(\lb)}-\lb F(\lb)$, which is
  non-positive, since $\tm^{F(\lb)} = \int x \, \mu^{F(\lb)}(\dx) \le
  \int \lb\, \mu^{F(\lb)}(\dx)$. Hence there is a unique $\ps$
  satisfying \eqref{eq:11}.

  Now define a measure
  \begin{equation*}
    \chi = \mu_{F(\lb)}^{F(\ub-)} + \lambda \delta_{\lb} +
    (\gamma-\lambda) \delta_{\ub}.
  \end{equation*}
  From \eqref{eq:10} it follows that $\chi$ is centered, and we embed this initially. The mass which arrives at $\ub$ will then run to the measure
  \begin{equation*}
    \nu = \frac{(\gamma-\lambda -(1-\ps))\delta_{\lb} + \mu_{\ps}}{\gamma-\lambda}
  \end{equation*}
  which has mean $\ub$ by the following computation:
  \begin{align*}
    (\gamma - \lambda) \int x \, \nu(\dx) & = \lb (\gamma-\lambda
    -(1-\ps)) + \tm_{\ps}\\
    & = \lb (\gamma-\lambda -(1-\ps)) - \tm_{F(\ub-)}^{\ps} -
    \tm_{F(\lb)}^{F(\ub-)} - \tm^{F(\lb)} \\
    & = \lb (\gamma-\lambda -(1-\ps)) - \lb(\xi + \ps - F(\ub-))  +
    \lambda \lb + (\gamma - \lambda) \ub\\
    & = \lb (\gamma -1 -\xi + F(\ub-)) + \ub (\gamma-\lambda).
  \end{align*}
  Here we have used \eqref{eq:10}, \eqref{eq:11} and the fact that $\xi = F(\lb)$. From the definition of $\gamma$ in \eqref{eq:16}, the desired conclusion follows.

  Finally, we embed the remaining part of $\mu$ from the mass that finishes at $\lb$ after either the first or second step, which has total probability $\gamma - \lambda + \ps - 1 + \lambda = \xi + \ps - F(\ub-)$. Set
  \begin{equation}\label{eq:19}
    \sigma = \frac{\mu^{\xi} + \mu_{F(\ub-)}^{\ps}}{\xi + \ps - F(\ub-)},
  \end{equation}
  and $\sigma$ has mean $\lb$:
  \begin{align*}
    (\xi + \ps - F(\ub-)) \int x \, \sigma(\dx) & = \tm^{\xi} +
    \tm_{F(\ub-)}^{\ps}\\
    & = \lb(\xi + \ps - F(\ub-))
  \end{align*}
  by \eqref{eq:11}. The final stopping time will be of the same form
  both in this case and in the case where \eqref{eq:8} holds, and when
  \eqref{eq:8} fails but \eqref{eq:12} holds. So before constructing
  the embedding, we give a description of the relevant measures in the
  second case.

  Suppose \eqref{eq:8} fails, but \eqref{eq:12} holds. Then in a
  similar manner to above, we can find $\xi \in
  (F(\lb),\rho_-(0)\wedge F(\ub-))$
  such that \eqref{eq:13} holds. Define
  \begin{equation*}
    \chi  = \mu_{\xi}^{F(\ub-)} + (1-F(\ub-)+\xi) \delta_{\ub}
  \end{equation*}
  and choose $\ps$ as before as the solution to \eqref{eq:11}.
  Then set
  \begin{equation*}
    \nu = \frac{(\ps-F(\ub-) + \xi) \delta_{\lb} + \mu_{\ps}}{1-F(\ub-)+\xi}
  \end{equation*}
  and we verify that $\nu$ has mean $\ub$:
  \begin{align*}
    (1-F(\ub-)+\xi) \int x \, \nu(\dx) & = \lb (\ps - F(\ub-)+\xi) +
  \tm_{\ps}\\
  & = \lb ( \ps - F(\ub-) + \xi) - \tm_{F(\ub-)}^{\ps} -
  \tm_{\xi}^{F(\ub-)}-\tm^{\xi}\\
  & = \lb(\ps-F(\ub-) + \xi) - \lb(\xi+\ps-F(\ub-)) + \ub(1-F(\ub-) + \xi)\\
  & = \ub(1-F(\ub-)+\xi).
  \end{align*}
  Finally, setting $\sigma$ as in \eqref{eq:19} we again have $\sigma$
  with mean $\lb$.

  In both cases, we construct an embedding as follows: let $\tau_1$ be
  a UI embedding of $\chi$ (starting from $0$). Then let $\tau_2$
  be a UI embedding of $\nu$ (starting from
  $\ub$). Finally, we let $\tau^3$ be a
  UI embedding of $\sigma$ (starting from $\lb$). We then define the
  complete embedding by:
  \begin{equation*}
    \begin{split}
      \tau:=\ &\tau_1\indic{B_{\tau_1}\in
        (\lb,\ub)}\\
      &+\tau_2\circ\tau_1\indic{B_{\tau_1}=\ub}
      \indic{B_{\tau_2\circ\tau_1}>\lb}+\\
      &+\tau_3\circ\Big(\tau_1\indic{B_{\tau_1}=\lb}
      +\tau_2\circ\tau_1\indic{B_{\tau_1}=\ub}
      \indic{B_{\tau_2\circ\tau_1}=\lb}\Big),
    \end{split}
  \end{equation*}
  and it follows from our construction that $\tau$ is a UI embedding
  of $\mu$ which moreover satisfies $\dbmps =
  \lG_{II}(\ps,\xi)$.

  Suppose finally we are in case \fbox{III}, so that \eqref{eq:7}, \eqref{eq:8}
  and \eqref{eq:12} all fail. Then there exists $\ps \in
  [\rho_-(0),F(\ub-))$ such that \eqref{eq:14} holds.

  Define the probability measure
  \begin{equation*}
    \chi = \mu_{\ps}^{F(\ub-)} + (1-F(\ub-)-\ps) \delta_{\ub},
  \end{equation*}
  which has mean $0$ by the definition of $\ps$. Define also
  \begin{equation*}
    \nu = \frac{\rho_-(0) \delta_{\lb} + \mu_{\rho_-(0)}^{\ps} + \mu_{F(\ub-)}}{1-F(\ub-)+\ps}
  \end{equation*}
  and we confirm that $\nu$ has mean $\ub$:
  \begin{align*}
    (1-F(\ub-)+\ps) \int x \, \nu(dx) & = \tm^{\rho_-(0)} +
    \tm_{\rho_-(0)}^{\ps} + \tm_{F(\ub-)}\\
    & = \tm^{\ps} + \tm_{F(\ub-)} \\
    & = -\tm_{\ps}^{F(\ub-)}\\
    & = \ub(1-F(\ub-)+\ps).
  \end{align*}
  Finally, any mass which is at $\lb$ we finally embed to the measure
  $\sigma = (\rho_-(0))^{-1}\mu^{\rho_-(0)}$. That is, we define the
  stopping times $\tau_1$ which is a UI embedding of $\chi$ starting
  at $0$. Then let $\tau_2$ be a UI embedding of $\nu$, given initial
  value $\ub$, and $\tau_3$ an embedding of $\sigma$ given initial
  value $\lb$. Finally, we define
  $$\tau:=\tau_1\indic{B_{\tau_1}\neq
    \ub}+\tau_2\circ\tau_1\indic{B_{\tau_1}=\ub}
  \indic{B_{\tau_2\circ\tau_1} > \lb}+ \tau_3 \circ\tau_2\circ\tau_1\indic{B_{\tau_1}=\ub}
  \indic{B_{\tau_2\circ\tau_1} = \lb},$$ to get a UI embedding of
  $\mu$ in $B$. Furthermore, it follows from the construction that
  $\dbmps = \lG_{III}(\ps,\rho_-(0))$.
\end{proof}

\section{On joint distribution of the maximum and minimum of a continuous UI martingale}
\label{ap:mart}

We turn now to studying the properties of joint distribution of the maximum and minimum of a continuous UI martingale. As previously, $(M_t:0\leq t\leq \infty)$ is a uniformly integrable continuous
martingale. We let $\mu$ be its terminal distribution, $\mu\sim M_\infty$, and recall that $-\infty\leq
\lmu<\rmu\leq \infty$ are the bounds of the support of $\mu$, i.e.\
$[\lmu,\rmu]$ is the smallest interval with
$\mu([\lmu,\rmu])=1$. Using Theorems~\ref{thm:upper_price_mixed} and
\ref{thm:lower_price_mixed} above, as well as existing results, we study the functions
\begin{align}
  \label{eq:20}
  p(\lb,\ub) & = \Pr\left(\iM_\infty > \lb\textrm{ and }\sM_\infty <\ub
  \right)\\
  q(\lb,\ub) & = \Pr\left(\iM_\infty > \lb\textrm{ and }\sM_\infty \ge \ub
  \right)\label{eq:21}\\
  r(\lb,\ub) &  = \Pr\left(\iM_\infty\le \lb\textrm{ and }\sM_\infty \ge \ub
  \right)\label{eq:22}
\end{align}
for $\lb\leq 0 \leq \ub$. 
Note that with no restrictions on $M_0$, when looking at extrema of the functions above, it is enough to consider $M_0$ a constant (e.g.\ when maximising $r$) or $M_0\equiv M_\infty$ (e.g.\ when minimising $r$). The latter is degenerate and henceforth we assume $M_0$ is a constant a.s. Further, as our results are translation invariant, we may and will take $M_0=0$ a.s. It follows that $\mu$ is centred.

It follows from Dambis, Dubins-Schwarz Theorem that $M$ is a (continuous) time change of Brownian motion, i.e.\ we can write $M_t=B_{\tau_t}$, $t\leq \infty$, for some Brownian motion and an increasing family of stopping times $(\tau_t)$ with $B_{\tau_\infty}\sim M_\infty$, $(B_{t\land \tau_\infty}:t\geq 0 )$ UI and $\sM_\infty=\overline{B}_{\tau_\infty}$, $\iM_\infty=\underline{B}_{\tau_\infty}$. In consequence, the problem reduces to studying the maximum and minimum of Brownian motion stopped at $\tau=\tau_\infty$, which is a solution the Skorokhod embedding problem. We can deduce results about the optimal properties of the martingales from corresponding results about Skorokhod embeddings. Our first result concerns the embeddings of Perkins and the `tilted-Jacka' construction, which we now recall using the notation established previously. These constructions have been considered in \cite{Cox:2011ab}, and we will need some results from this paper; however both constructions have a long history --- see for example \cite{Perkins:86,Cox:2004aa,Jacka:88,Cox:2005aa}.  For the Perkins embedding we define\footnote{Strictly, we only consider the case where $\mu(\{0\}) = 0$. If this is not the case, then the optimal embedding requires independent randomisation to stop some mass at zero initially.}
\begin{equation}
  \label{eq:24}
  \begin{split}
    \gp(p)=q \text{ where $q$ solves }  &\tm^{q} + \tm_{p} = (1-p+q)
    F(p), \quad p > F(0)\\
    \gp(q)=p \text{ where $p$ solves }  &\tm^{q} + \tm_{p} = (1-p+q)
    F(q), \quad q < F(0-).
  \end{split}
\end{equation}
The stopping time $\tau_P$ is then defined via:
\begin{equation}
  \label{eq:25}
  \tau_P = \inf\{ t \ge 0: F(B_t) \not\in (\gp(F(\uB_t)),\gm(F(\lB_t)))\}.
\end{equation}

In a similar spirit, the tilted-Jacka construction is given as follows. Choose $\ps\in [0,1]$ such that $(\lb-m^{\ps})(m_{\ps}-\ub)\ge 0$ --- this is always possible, since we can always find $\ps$ such that $m^{\ps} = \lb$ say.  Then set $\chi = \ps \delta_{m^{\ps}} + (1-\ps) \delta_{m_{\ps}}$. The construction is as follows: we first embed the distribution $\chi$, then, given we hit $m^{\ps}$, we embed $\mu^{\ps}$ using the reversed Az\'ema-Yor construction (c.f.~\cite{Obloj:04b}); if we hit $m_{\ps}$ then we embed $\mu_{\ps}$ using the Az\'ema-Yor construction.

Finally, we observe that both cases give rise to martingales with
certain optimality properties using the fact that the stopped Brownian
motion is a continuous martingale.

\begin{prop}\label{prop:M_bounds}
  We have the following properties:
  \begin{enumerate}
  \item $p(0,\ub)=0 = p(\lb,0)$, $q(0,\ub) = 0 = q(\lb,\rmu)$ and
    $r(\lmu,\ub) = 0 = r(\lb,\rmu)$;
  \item $p(\lb,\ub)=1$ on $[-\infty,\lmu)\times(\rmu,\infty]$,
    $q(\lb,\ub) = 1$ on $[-\infty,\lmu) \times \{0\}$, and $r(0,0) = 1$;
  \item $p$ and $q$ are non-increasing in $\lb\in (\lmu,0)$ and
    $p$ is non-decreasing in $\ub\in (0,\rmu)$; $r$ is non-decreasing
    in $\lb \in (\lmu,0)$ and $q$ and $r$ are non-decreasing in $\ub
    \in (0,\rmu)$;
  \item for $\lmu\leq \lb<0<\ub\leq\rmu$ we have
    \begin{equation}\label{eq:uimart_bound}
      \Pr\big(\underline{B}_{\tau_J}>\lb\textrm{ and
      }\overline{B}_{\tau_J} <\ub \big)\leq p(\lb,\ub)\leq
      \Pr\big(\underline{B}_{\tau_P}>\lb\textrm{ and
      }\overline{B}_{\tau_P} <\ub \big),
    \end{equation}
    where $(B_t)$ is a standard Brownian motion with $B_0=0$,
    $\tau_P$ is the Perkins stopping time \cite[(4.4)]{Cox:2011ab} embedding
    $\mu$ and $\tau_J$ is the `tilted-Jacka' stopping time \cite[(4.6)]{Cox:2011ab},
    for barriers $(\lb,\ub)$, embedding $\mu$;
  \item for $\lmu\leq \lb<0<\ub\leq\rmu$, the lower bound on
    $q(\lb,\ub)$ is given by \eqref{eq:generalbound_dbmp}, and the
    upper bound is given by \eqref{eq:generalbound_dbmp2}. Moreover
    these bounds are attained by the constructions in
    Theorems~\ref{thm:upper_price_mixed} and
    \ref{thm:lower_price_mixed} respectively;
  \item for $\lmu\leq \lb<0<\ub\leq\rmu$, the lower bound on
    $r(\lb,\ub)$ is given by Proposition~2.3 of \cite{Cox:2011aa}, and
    the upper bound is given by Proposition~2.1 of \cite{Cox:2011aa}.
    Moreover these bounds are attained by the constructions in
    Theorems~2.4 and 2.2 of \cite{Cox:2011aa} respectively.
  \end{enumerate}
\end{prop}
The first three assertions of the proposition are clear. Assertion
$(iv)$ is a reformulation of Lemmas~4.2 and 4.3 of \cite{Cox:2011ab}
--- it suffices to note that $(B_{t\wedge\tau_J})$,
$(B_{t\wedge\tau_P})$, $(M_t)$ are all UI martingales starting at $0$
and with the same terminal law $\mu$ for $t=\infty$. Likewise, part $(vi)$
is a reinterpretation of the results of \cite{Cox:2011aa}. We note that therein the results were formulated for the case of non-atomic $\mu$. They extend readily, with methods used in Section \ref{sec:proofs} above, specifically by characterising the stopping distributions via quantiles of the underlying measures, to the general case.

We can think of any of the functions $p(\cdot,\cdot), q(\cdot,\cdot)$, and $r(\cdot,\cdot)$ as a surface defined over the quarter-plane $[-\infty,0]\times [0,\infty]$. Proposition \ref{prop:M_bounds} describes boundary values of the surface, monotonicity properties and gives an upper and a lower bound on the surface. However we note that --- most obviously in $(iv)$ --- there is a substantial difference between the bounds linked to the fact that $\tau_P$ does not depend on $(\lb,\ub)$ while $\tau_J$ does. In consequence, the upper bound is attainable: there is a martingale $(M_t)$, namely $M_t=(B_{t\wedge\tau_P})$, for which $p$ is equal to the upper bound for all $(\lb,\ub)$. In contrast a martingale $(M_t)$ for which $p$ would be equal to the lower bound does not exist. For the martingale $M_t=(B_{t\wedge\tau_J})$, where $\tau_J$ is defined for some pair $(\lb,\ub)$, $p$ will attain the lower bound in some neighbourhood of $(\lb,\ub)$ which will be strictly contained in $(\lmu,0)\times (0,\rmu)$. More generally, the latter case is more typical of all the constructions which are used in the result; however, with some careful construction, it seems likely that one can usually find a construction which will be optimal for all values of $(\lb, \ub)$ which lie in some small open set (for example, this is true of the tilted-Jacka construction), but there will be limits on how large the region on which a given construction is optimal can be made.

We now give a result which provides some further insight into the structure of the bounds discussed above. In particular, we can show some finer properties of the functions $p,q,r$ and their upper and lower bounds. We state and prove the result for the function $p$, but the corresponding versions for $q$ and $r$ will follow in a clear manner.

\begin{theorem} \label{thm:boundstructure} The function $p(\lb,\ub)$ is c\`agl\`ad in $\ub$ and c\`adl\`ag in $\lb$. Moreover, if $p$ is discontinuous at $(\lb,\ub)$, then $\mu$ must have an atom at one of $\lb$ or $\ub$. Further:
  \begin{enumerate}
  \item if there is a discontinuity at $(\lb,\ub)$ of the form:
    \[
    \limsup_{w \to \ub} p(\lb,w) > p(\lb,\ub)
    \]
    then the function $g$ defined by
    \[
    g(u) = \limsup_{w \to \ub} p(u,w) - p(u,\ub), \qquad u \le
    \lb
    \]
    is non-increasing.
  \item if there is a discontinuity at $(\lb,\ub)$ of the form:
    \[
    \limsup_{u \to \lb} p(u,\ub) > p(\lb,\ub)
    \]
    then the function $h$ defined by
    \[
    h(w) = \limsup_{u \to \lb} p(u,w) - p(\lb,w), \qquad w \ge
    \ub
    \]
    is non-decreasing.
  \end{enumerate}
  And, at any discontinuity, we will be in at least one of the above
  cases.

  In addition the lower bound (corresponding to the
  tilted-Jacka construction) is continuous in $(\lmu,0)\times
  (0,\rmu)$, and continuous at the boundary ($\ub = \rmu$ and $\lb
  = \lmu$) unless there is an atom of $\mu$ at either $\rmu$ or
  $\lmu$, while the upper bound (which corresponds to the Perkins
  construction) has a discontinuity corresponding to every atom of
  $\mu$.
\end{theorem}

\begin{remark}
  \begin{enumerate}
  \item   Considering $q$ instead of $p$, the function will be c\`adl\`ag in
  both arguments, and the directions of the convergence results needs
  to be adapted suitably. We also observe that discontinuities in the
  upper bound occur only if there is an atom of $\mu$ at $\lb$, {\it
    and} we are in case \fbox{I} of
  Theorem~\ref{thm:upper_price_mixed}. Similarly, there is a
  discontinuity in the lower bound at $\ub$ if there is an atom of
  $\mu$ at $\ub$, and we are in either of cases \fbox{II} or
  \fbox{III} of Theorem~\ref{thm:lower_price_mixed}.

\item Considering $r$ instead of $p$, the function will be c\`agl\`ad
  in $\lb$ and c\`adl\`ag in $\ub$. We also observe that
  discontinuities in the upper bound never occur, while there are
  discontinuities in the lower bound at $\ub$ and/or $\lb$ if there is an atom of
  $\mu$ at either of these values.
  \end{enumerate}
\end{remark}

Before we prove the above result, we note the following useful result,
which is a simple consequence of the martingale property:

\begin{prop} \label{prop:supatom} Suppose that $(M_t)_{t \ge 0}$ is a
  UI martingale with $M_\infty \sim \mu$. Then $\Pr(\sM_{\infty} =
  \ub) >0$ implies $\mu(\{\ub\}) \ge \Pr(\sM_{\infty} = \ub)$ and
  \[
  \{ \sM_{\infty} = \ub \} = \{ M_t = \ub, \ \forall t \ge H_{\ub}\}
  \subseteq \{ M_\infty = \ub\} \quad \as{}.
  \]
\end{prop}

\begin{proof}[Proof of Theorem~\ref{thm:boundstructure}]
  We begin by noting that by definition of $p(\lb,\ub)$, we
  necessarily have the claimed continuity and limiting
  properties. Further,
  \[
  \liminf_{(s,v) \to (u,w)} p(s,v) \ge \Pr(\iM_\infty > \lb
  \mbox{ and } \sM_\infty < \ub)
  \]
  and
  \[
  \limsup_{(s,v) \to (u,w)} p(s,v) \le \Pr(\iM_\infty \ge \lb
  \mbox{ and } \sM_\infty \le \ub).
  \]
  It follows that the function $p$ is continuous at $(\lb,\ub)$ if
  $\Pr(\iM_\infty = \lb) = \Pr(\sM_\infty = \ub) = 0$. By
  Proposition~\ref{prop:supatom}, this is true when $\mu(\{\ub,\lb\})
  = 0$.

  Note that we can now see that at a discontinuity of $p$, we must be
  in at least one of the cases (i) or (ii). This is because
  discontinuity at $(\lb,\ub)$ is equivalent to
  \[
  \Pr(\iM_\infty \ge \lb \mbox{ and } \sM_\infty \le \ub) >
  \Pr(\iM_\infty > \lb \mbox{ and } \sM_\infty < \ub),
  \]
  from which we can deduce that at least one of the events
  \[
  \{\iM_\infty > \lb \mbox{ and } \sM_\infty = \ub\}, \quad
  \{\iM_\infty = \lb \mbox{ and } \sM_\infty < \ub\}, \quad
  \{\iM_\infty = \lb \mbox{ and } \sM_\infty = \ub\}
  \]
  is assigned positive mass. However, by
  Proposition~\ref{prop:supatom} the final event implies both
  $M_\infty = \lb$ and $M_\infty = \ub$ which is
  impossible. Consequently, at least one of the first two events must
  be assigned positive mass, and these are precisely the cases (i) and
  (ii).

  Consider now case (i). We can rewrite the statement as: if $g(\lb) >
  0$, then $g(u)$ is decreasing for $u<\lb$. Note however that
  \begin{eqnarray*}
    g(u) & = & \Pr(\iM_\infty > u \mbox{ and } \sM_\infty \le \ub) -
    \Pr(\iM_\infty > u \mbox{ and } \sM_\infty < \ub) \\
    & = & \Pr(\iM_\infty > u \mbox{ and } \sM_\infty = \ub)
  \end{eqnarray*}
  which is clearly non-increasing in $u$. {In fact, provided that
    $g(\lb)<\Pr(\sM_\infty=\ub)$, it follows from \eg{}
    \cite[Theorem~4.1]{Rogers:93} that $g$ is strictly decreasing for
    $\lb>u>\sup\{u\geq -\infty: g(u)=\Pr(\sM_\infty=\ub)\}$.  A
    similar proof holds in case (ii).}

  We now consider the lower bounds corresponding to the tilted-Jacka
  construction. We wish to show that
  \[
  \Pr(\iM_\infty \ge \lb \mbox{ and } \sM_\infty \le \ub) =
  \Pr(\iM_\infty > \lb \mbox{ and } \sM_\infty < \ub),
  \]
  for any $(\ub,\lb)$ except those excluded in the statement of the
  theorem. We note that it is sufficient to show that $\Pr(\iM_\infty
  = \lb) = \Pr(\sM_\infty = \ub) = 0$, and by
  Proposition~\ref{prop:supatom} it is only possible to have an atom
  in the law of the maximum or the minimum if the process stops at the
  maximum with positive probability; we note however that the stopping
  time $\tau_J$, due to properties of the Az\'ema-Yor embedding
  precludes such behaviour except at the points $\lmu, \rmu$.

  Considering now the Perkins construction, we note from
  \eqref{eq:25} and the fact that the function $\gp$ is
  decreasing, that we will stop at $\lb$ only if $\gp(F(\sM_t))
  = \lb$ and $M_t = \iM_t = \lb$. It follows from \eqref{eq:24}
  that there is a range of values $(\ub_*,\ub^*)$ for which
  $\gp(F(b)) = \lb$, and consequently, we must have $h(b) =
  \Pr(\iM_\infty = \lb, \sM_\infty < b)$ increasing in $b$ as $b$ goes
  from $\ub_*$ to $\ub^*$, with $h(\ub_*) = \Pr(\iM_\infty = \lb,
  \sM_\infty < \ub_*)=0$ and $h(\ub^*) = \Pr(\iM_\infty = \lb,
  \sM_\infty < \ub^*)=\mu(\{\lb\})$.\footnote{In fact, as above, it
    follows from \eg{} \cite[Theorem~2.2]{Rogers:93} that the maximum
    must have a strictly positive density with respect to Lebesgue
    measure, and therefore that the function $h$ is strictly
    increasing between the points $\ub_*$ and $\ub^*$.} Similar
  results for the function $g$ also follow.
\end{proof}

\section*{Conclusions}
In this paper, we studied the possible joint distributions of
$(\sM_\infty,\iM_\infty)$ given the law of $M_\infty$, and were able to obtain
number of qualitative properties and sharp quantitative bounds. It follows from
our results that the interaction between the maximum and minimum is highly
non-trivial which makes the pair above much harder to study than
$\overline{M}_\infty$ and $\iM_\infty$ on their own. This is best seen in the
case of Brownian motion where $\overline{B}_t$ has an easily accessible
distribution while the description of the joint distribution of
$(\underline{B}_t,\overline{B}_t)$ is much more involved. A further natural
question arising from our work is to characterise the joint distributions of the
joint distributions of the triple $(M_\infty,\sM_\infty,\iM_\infty)$.  At
present it is not clear to us if, and to what extent, a complete
characterisation of the possible joint distributions of this triple, in the
spirit of Rogers \cite{Rogers:93} and Vallois \cite{Vallois:93}, is feasible.
It remains an open and challenging problem.

\bibliographystyle{alpha} \bibliography{doubleexit}

\newcommand{\etalchar}[1]{$^{#1}$}
\begin{thebibliography}{CEKO12}

\bibitem[ABP{\etalchar{+}}13]{Acciaio:2013ab}
B.~Acciaio, M.~Beiglbock, F.~Penkner, W.~Schachermayer, and J.~Temme.
\newblock A trajectorial interpretation of doob's martingale inequalities.
\newblock {\em Ann. Appl. Prob}, 23(4):1494--1505, 2013.

\bibitem[AY79a]{AzemaYor:79}
J.~Az{\'e}ma and M.~Yor.
\newblock \abc {U}ne solution simple au probl\`eme de {S}korokhod.
\newblock In {\em S\'eminaire de Probabilit\'es, XIII (Univ. Strasbourg,
  Strasbourg, 1977/78)}, volume 721 of {\em Lecture Notes in Math.}, pages
  90--115. Springer, Berlin, 1979.

\bibitem[AY79b]{AzemaYor:79b}
J.~Az{\'e}ma and M.~Yor.
\newblock Le probl\`eme de {S}korokhod: compl\'ements \`a ``{U}ne solution
  simple au probl\`eme de {S}korokhod''.
\newblock In {\em S\'eminaire de Probabilit\'es, XIII (Univ. Strasbourg,
  Strasbourg, 1977/78)}, volume 721 of {\em Lecture Notes in Math.}, pages
  625--633. Springer, Berlin, 1979.

\bibitem[BD63]{BD63}
D.~Blackwell and L.~E. Dubins.
\newblock A converse to the dominated convergence theorem.
\newblock {\em Illinois J. Math.}, 7:508--514, 1963.

\bibitem[BHR01]{Brown:01b}
H.~Brown, D.~Hobson, and L.~C.~G. Rogers.
\newblock Robust hedging of barrier options.
\newblock {\em Math. Finance}, 11(3):285--314, 2001.

\bibitem[CEKO12]{CarraroElKarouiObloj:09}
L.~Carraro, N.~El~Karoui, and J.~Ob{\l}\'{o}j.
\newblock On {A}z{\'e}ma-{Y}or processes, their optimal properties and the
  {B}achelier-{D}rawdown equation.
\newblock {\em The Annals of Probability}, 40(1):372--400, 2012.

\bibitem[CH04]{Cox:2004aa}
A.~M.~G. Cox and D.~G. Hobson.
\newblock An optimal {S}korokhod embedding for diffusions.
\newblock {\em Stochastic Process. Appl.}, 111(1):17--39, 2004.

\bibitem[CH05]{Cox:2005aa}
A.~M.~G. Cox and D.~G. Hobson.
\newblock Skorokhod embeddings, minimality and non-centred target
  distributions.
\newblock {\em Probability Theory and Related Fields}, 135:395--414, 2005.

\bibitem[CO11a]{Cox:2011aa}
A.~M.~G. Cox and J.~Ob{\l}{\'o}j.
\newblock Robust hedging of double touch barrier options.
\newblock {\em {SIAM} Journal on Financial Mathematics}, 2(1):141--182, 2011.

\bibitem[CO11b]{Cox:2011ab}
A.~M.~G. Cox and J.~Ob{\l}{\'o}j.
\newblock Robust pricing and hedging of double no-touch options.
\newblock {\em Finance and Stochastics}, 15(3):573--605, 2011.

\bibitem[DG78]{DubinsGilat:78}
L.~E. Dubins and D.~Gilat.
\newblock On the distribution of maxima of martingales.
\newblock {\em Proc. Amer. Math. Soc.}, 68(3):337--338, 1978.

\bibitem[DR14]{Duembgen:2014aa}
M.~Duembgen and L.~C.~G. Rogers.
\newblock The joint law of the extrema, final value and signature of a stopped
  random walk.
\newblock arXiv:1403.0220 [math.PR], 2014.

\bibitem[GM88]{GM88}
D.~Gilat and I.~Meilijson.
\newblock A simple proof of a theorem of {B}lackwell \& {D}ubins on the maximum
  of a uniformly integrable martingale.
\newblock In {\em {S}\'eminaire de {P}robabilit\'es, XXII}, volume 1321 of {\em
  Lecture Notes in Math.}, pages 214--216. Springer, Berlin, 1988.

\bibitem[Hob98]{Hobson:98b}
D.~G. Hobson.
\newblock Robust hedging of the lookback option.
\newblock {\em Finance Stoch.}, 2(4):329--347, 1998.

\bibitem[Hob10]{Hobson:10}
D.~Hobson.
\newblock {The Skorokhod Embedding Problem and Model-Independent Bounds for
  Option Prices}.
\newblock In R.A. Carmona, E.~{\c{C}}inlar, I.~Ekeland, E.~Jouini, J.A.
  Scheinkman, and N.~Touzi, editors, {\em Paris-Princeton Lectures on
  Mathematical Finance 2010}, volume 2003 of {\em Lecture Notes in Math.},
  pages 267--318. Springer, 2010.

\bibitem[Jac88]{Jacka:88}
S.~D. Jacka.
\newblock Doob's inequalities revisited: a maximal {$H\sp 1$}-embedding.
\newblock {\em Stochastic Process. Appl.}, 29(2):281--290, 1988.

\bibitem[KR90]{KR90}
R.~P. Kertz and U.~R{\"o}sler.
\newblock Martingales with given maxima and terminal distributions.
\newblock {\em Israel J. Math.}, 69(2):173--192, 1990.

\bibitem[KR92]{KertzRosler:92b}
R.~P. Kertz and U.~R{\"o}sler.
\newblock Stochastic and convex orders and lattices of probability measures,
  with a martingale interpretation.
\newblock {\em Israel J. Math.}, 77(1-2):129--164, 1992.

\bibitem[KR93]{KertzRosler:93}
R.~P. Kertz and U.~R{\"o}sler.
\newblock Hyperbolic-concave functions and {H}ardy-{L}ittlewood maximal
  functions.
\newblock In {\em Stochastic inequalities ({S}eattle, {WA}, 1991)}, volume~22
  of {\em IMS Lecture Notes Monogr. Ser.}, pages 196--210. Inst. Math.
  Statist., Hayward, CA, 1993.

\bibitem[Ob{\l}04]{Obloj:04b}
J.~Ob{\l}{\'o}j.
\newblock The {S}korokhod embedding problem and its offspring.
\newblock {\em Probab. Surv.}, 1:321--390 (electronic), 2004.

\bibitem[Ob{\l}10]{Obloj:EQF}
J.~Ob{\l}{\'o}j.
\newblock The {S}korokhod embedding problem.
\newblock In Rama Cont, editor, {\em Encyclopedia of Quantitative Finance},
  pages 1653--1657. Wiley, 2010.

\bibitem[Per86]{Perkins:86}
E.~Perkins.
\newblock The {C}ereteli-{D}avis solution to the ${H}\sp 1$-embedding problem
  and an optimal embedding in {B}rownian motion.
\newblock In {\em Seminar on stochastic processes, 1985 (Gainesville, Fla.,
  1985)}, pages 172--223. Birkh\"auser Boston, Boston, MA, 1986.

\bibitem[Rog93]{Rogers:93}
L.~C.~G. Rogers.
\newblock The joint law of the maximum and terminal value of a martingale.
\newblock {\em Probab. Theory Related Fields}, 95(4):451--466, 1993.

\bibitem[Val93]{Vallois:93}
P.~Vallois.
\newblock On the joint distribution of the supremum and terminal value of a
  uniformly integrable martingale.
\newblock In {\em Stochastic processes and optimal control (Friedrichroda,
  1992)}, volume~7 of {\em Stochastics Monogr.}, pages 183--199. Gordon and
  Breach, Montreux, 1993.

\end{thebibliography}

\end{document}